\documentclass[12pt,leqno]{amsart}
\usepackage{amsmath,amsthm}
\usepackage{amsfonts,amssymb}
\usepackage{amsrefs}
\usepackage{graphicx}
\newtheorem{theorem}{Theorem}
\newtheorem{lemma}[theorem]{Lemma}
\newtheorem{proposition}[theorem]{Proposition}

\begin{document}
\title[Sticky particles]{Sticky particles and stochastic flows}
\author{Jon Warren}
\address{Department of Statistics, University of Warwick, Coventry CV4 7AL, UK}
\email{j.warren@warwick.ac.uk}
\date{}

\begin{abstract}
Gaw\c{e}dzki and Horvai  have studied a model for the motion of particles carried in  a turbulent fluid and  shown that in a limiting regime with  low levels of viscosity and molecular diffusivity, pairs of particles exhibit the phenomena of stickiness when they meet. In this paper we characterise the  motion of an arbitrary number of particles in a   simplified version of their model.
\end{abstract}


\maketitle

\section{Introduction}

The motivation for this paper comes from a work by Gaw\c{e}dzki and Horvai, \cite{gh}, in which the authors study a model for the motion of particles carried in  a turbulent fluid. The trajectories of two distinct particles $\bigl(X_{1}(t), t\geq 0\bigr) $ and  $\bigl( X_{2}(t), t \geq 0\bigr)$ are each described by a Brownian motion in ${\mathbf R}^d$ with a covariance of the form
\begin{equation}
\langle X_{1}, X_{2} \rangle (t)=\int_0^t \psi\bigl(  X_{1}(s)- X_{2}(s)\bigr) ds.
\end{equation}
The $d \times d$ matrix valued function $\psi$ is invariant under the natural action of the orthogonal group and consequently the inter-particle distance $\| X_{1}(t)- X_{2} (t)\|$ is  a diffusion process on ${\mathbf R}_+$. For different choices of the covariance function $\psi$, different qualitative behaviours are observed, and these correspond to different boundary conditions at $0$ for the diffusion describing the inter-particle distance. See also Le Jan and Raimond \cite{lejan1} for a description of these phases. Gaw\c{e}dzki and Horvai study the case where $0$ is both a entrance and exit boundary point, and  the function $\psi$ is not  smooth at the origin. They then introduce a  viscosity effect  acting at small scales  by replacing $\psi$ by a smooth covariance function  obtained by  smoothing $\psi$ in  a neighbourhood of the origin. Particles moving in this regularized flow never meet, and $0$ is now a natural boundary point for the diffusion describing the inter-particle distance. They then further vary the model and consider particles whose motion is affected by molecular diffusivity, modelled by adding, for each particle, a small independent Brownian perturbation to the motion of the flow. If the additional  diffusivity  and the scale  at which viscosity acts both are taken to zero in an appropriate balance then Gaw\c{e}dzki and Horvai show that the inter-particle distance $\| X^{(1)}(t)- X^{(2)} (t)\|$ converges to a diffusion on ${\mathbf R}_+$ with the boundary point being sticky: that is a regular boundary point at which the diffusion spends a strictly positive amount of time.

Sticky boundary behaviour was first identified by Feller, as described in the article \cite{pes}. Subsequently  the process which is a Brownian motion on ${\mathbf R}_+$ with a sticky boundary at $0$ was studied as an example of a stochastic differential equation with no strong solution, see Chitashvili, \cite{ch} and 
Warren \cite{w1}, and  recent work by Engelbert and  Peskir \cite{ep} and Bass \cite{bass}. Stochastic flows in which the inter-particle distance evolves as a sticky Brownian motion have been studied by Le Jan and Lemaire \cite{lejan3}, by Howitt and Warren \cite{hw1} and \cite{hw2}, and by Schertzer, Sun and Swart, \cite{sss}.

In this paper we study a simplification of the  Gaw\c{e}dzki-Horvai model. Our goal is to address, in this simplified setting, the question raised by Gaw\c{e}dzki and Horvai  of characterizing  the behaviour of $N$ particles. We take the dimension  of the underlying space to be $d=1$, and the motion of distinct particles, in the absence of viscosity or molecular diffusivity, to be given by Brownian motions which are independent of one another until the particles meet. 

Let $\psi$ be a real-valued, smooth, positive definite function on ${\mathbf R}$, satisfying $\psi(0)=1$, $|\psi(x)|<1$ for $x \neq 0$, and $\psi(x) \rightarrow 0$ as $|x|\rightarrow \infty$.
Define  the constant $a$, which we assume is strictly positive, via 
\begin{equation}
\label{ass}
 \frac{1-\psi(x)}{x^2}  \rightarrow a^2 \text{ as }
x \rightarrow 0.
\end{equation}


 For each $n$ there exists a smooth flow of
Brownian motions associated with the scaled covariance function $\psi(nx)$, the $N$ point motion of which has generator
\begin{equation}
\label{smoothgen}
\frac{1}{2}\sum_{i,j} \psi\bigl(n(x_i-x_j)\bigr))\frac{\partial^2 }{\partial
  x_ix_j}.
\end{equation}
As $n$ tends to infinity the covariance  functions $ \psi(nx)$ converge to the
singular covariance $ 1_0( x)$, and  correspondingly,  the $N$-point motions  associated with the flows
converge to systems of coalescing Brownian motions. 

Fix a constant $b >0$  and for $n \geq 1$, we define generators  
\begin{equation}
\label{pertgen}
{\mathcal G}^{N,n}= \frac{1}{2}\sum_{i, j}\psi\bigl(n(x_i-x_j)\bigr)\frac{\partial^2
}{\partial x_ix_j}+ \frac{b^2}{2n^2} \sum_{i}\frac{\partial^2
}{\partial x_i^2}
\end{equation} 
which are  perturbations of the generators
\eqref{smoothgen} by
addition of the Laplacian  with co-efficient   $b^2/2n^2$. This     works
against coalescence by giving each particle in the flow  a small
amount of independent diffusivity.
As a consequence paths of particles in the flow can cross and the
$N$-point motions are   no longer associated with flow of maps. 

 The  two effects:  approximating a coalescing flow by smooth flows, and
 adding diffusivity,  are in balance   as we pass to the limit,
as can be seen by the following analysis of the $2$-point motion.  
Let $(X_1,X_2)$ be the two point motion with generator ${\mathcal
G}^{2,n}$. It is enough to consider the difference $Z(t)=X_1-X_2(t)$
which is a diffusion on the real line in natural scale and with speed
measure
\begin{equation}
\label{speed}
m_n(dz)=\frac{dz}{1+b^2n^{-2}-\psi(nz)}
\end{equation}
As $n$ tends to infinity $m_n$ weakly converges to the measure
$m(dz)=dz+\theta^{-1}\delta_0(dz)$ where the constant $\theta$ is given
by
\begin{equation}
\label{theta}
\theta^{-1}= \int_{-\infty}^\infty \frac{dz}{b^2+a^2z^2}=\frac{\pi}{ab}.
\end{equation}
Thus the limiting diffusion describing $|X_1-X_2|$ is a  sticky Brownian
with the parameter $\theta$ describing the degree of stickiness at $0$, and the limit of the two point motion is determined by  this, together with $X_1$ and $X_2$ each being Brownian
motions.

This leaves open the limiting behaviour of the perturbed $N$-point
motions for $N \geq 3$.  Consistent families of diffusions in
${\mathbf R}^N$ whose components are Brownian motions evolving as
independent Brownian motions whenever they are unequal were studied
in \cite{hw1}. For such processes there are times at which many
co-ordinates co-incide and it is necessary to describe  the sticky 
behaviour at such times.   This is specified by families of
of non-negative co-efficients  $(\theta(k:l); k,l \geq 1)$.  Thinking of the $N$-point motion as a system of $N$
particles $\theta(k:l)$ gives the rate, in an excursion theoretic
sense,  at which a clump of  $k+l$ particles separates into two clumps
one consisting of $k$  particles and the other of $l$
particles. The result of this paper is the following
identification of these co-efficients for our model.

\begin{theorem}
\label{main}
The $N$ point motions with generators ${\mathcal G}^{N,n}$ converge in
law as $n$ tends to infinity to a family of sticky Brownian motions 
associated to the family of parameters  $(\theta(k:l); k,l \geq 1)$
given by
\[
\theta(k:l)= \frac{ab}{2\sqrt{\pi}}\int_{\mathbf R}\int_{{\mathbf R}^{k+l}}
\frac{e^{-\|x\|^2/2}}{(2\pi)^{(k+l)/2}} {\mathbf 1}
(x_1,x_2, \ldots x_k <z< x_{k+1}, \ldots, x_{k+l} )dxdz
\]
\end{theorem}

The form of the parameters $\theta(k:l)$ given in this result is highly suggestive of the underlying mechanisms at work. The variables $x_1, \dots,  x_{k+l}$  chosen according to a Gaussian measure can be thought of as the positions of  a cluster of $k+l$ particles  experiencing independent diffusivity, and the variable $z$  represents a ``singularity'' in  the underlying flow that causes the cluster to separate into two. Of course this is far from being rigorous.
 
To give  Theorem \ref{main} a precise meaning we must specify the law of the
family of sticky Brownian motions associated  to the family of
parameters  $(\theta(k:l); k,l \geq 1)$. We do this by means of a
well-posed martingale problem, following \cite{hw1}.  

Suppose $\bigl(\theta(k:l); k, l \geq 1)$ is a family of nonnegative
parameters satisfying the consistency property 
\begin{equation}
\theta(k:l)=\theta(k+1:l)+\theta(k:l+1)
\end{equation}
 For our purposes in this paper we may also assume  the symmetry
 $\theta(k:)= \theta(l:k)$.   
We now recall the main result from \cite{hw1} concerning the
characterization of consistent families of sticky Brownian motions.

We begin by  partitioning  ${\mathbf R}^N$ into cells. A cell $E \subset {\mathbf R}^N$ is determined by some weak total ordering 
$\preceq$ of the $\{1,2,\ldots N\}$ via 
\begin{equation}
E= \{ x \in {\mathbf R}^N: x_i \leq x_j \text{ if and only if } i \preceq j \}.
\end{equation}
Thus $\{x \in {\mathbf R}^3: x_1 =x_2=x_3\}$, $\{x \in {\mathbf R}^3: x_1 <x_2=x_3\}$ and $\{ x\in {\mathbf R}^3: x_1>x_2>x_3\}$ are 
three of the thirteen distinct cells into which ${\mathbf R}^3$ is partitioned.

Suppose that $I$ and $J$ are disjoint subsets of $\{1,2,\ldots ,N\}$
with both $I$ and $J$ non-empty. With such a pair we associate a 
vector $v=v_{I,J}$ belonging to ${\mathbf R}^N$ with components given by
\begin{equation}
\label{vdef}
v_i= \begin{cases} 0 & \text{ if $i \not\in I\cup J$,} \\
+1 & \text{ if $i \in I$,} \\
-1 & \text{ if $i \in J$.}
\end{cases}
\end{equation}
We  associate with each point $x\in {\mathbf R}^N$ certain vectors of this form. To this end note that each point $x \in {\mathbf 
R}^N$ determines a partition $\pi(x)$ of $\{1,2,\ldots N\}$ such that $i$ and $j$ belong to the same component of 
$\pi(x)$ if and only if $x_i=x_j$. Then to each  point  $x \in {\mathbf R}^N$ we associate the set of vectors, denoted by ${\mathcal V}(x)$, which 
consists of
every vector of the form $v=v_{IJ}$ where $I\cup J$ forms  one component of the partition $\pi(x)$.

Let $ L_N$ be the space of real-valued functions defined on ${\mathbf R}^N$ which are continuous, and whose restriction to each cell is 
given by a linear function. Given a set of parameters $\bigl(\theta(k:l); k,l \geq 0\bigr)$ we define the operator ${\mathcal A}^\theta_N$ from $L_N$ to the space of 
real valued functions on ${\mathbf R}^N$ which are constant on each cell by
\begin{equation}
{\mathcal A}^\theta_N f (x) = \sum_{ v \in {\mathcal V}(x)} \theta(v) \nabla_{v} f(x).
\end{equation}
Here on the righthandside $\theta(v)= \theta(k:l)$ where $k=|I|$ is the number of elements in $I$ and $l=|J|$ is the number of elements in $J$ for $I$ and $J$  determined by $v=v_{IJ}$.  $\nabla_v f(x)$ denotes  the (one-sided) gradient of $f$ in the direction $v$ at the point $x$, that is 
\begin{equation}
\nabla_v f(x)= \lim_{\epsilon \downarrow 0}  \frac{1}{\epsilon} \bigl( f( x+\epsilon v)- f(x) \bigr).
\end{equation}

We say an ${\mathbf R}^N$-valued stochastic process  $\bigl( X(t); t
\geq 0 \bigr)$ solves the ${\mathcal A}^\theta_N$-martingale problem
 if 
for  each $f \in L_N$,
\[
f\bigl(X(t)\bigr)- \int_0^t {\mathcal A}_N^\theta f \bigl(X(s)\bigr) ds \text{ is a martingale,}
\]
relative to some common filtration, and the bracket between co-ordinates $X_i$ and $X_j$ is given by
\[
\langle X_i,X_j\rangle (t)=\int_0^t {\mathbf 1}{( X_i(s)=X_j(s))} ds  \qquad \text{ for }  t\geq 0.
\]
In particular $\langle X_i \rangle(t)=t$.
According to the main result of  \cite{hw1},  for any  given starting
point $x \in {\mathbf R^N}$, a solution to the ${\mathcal A}^\theta_N$-martingale problem exists and its law is unique. It is
a process with this law that we refer to as a family of $N$ sticky
Brownian motions associated with  the parameters $\bigl( \theta(k:l);k, l\geq 1)$.

\section{Heuristic derivation of exit probabilities}

Let us write $\bigl(X(t); t \geq 0\bigr)$  for the co-ordinate process
on $N$ dimensional path space, and  we will write  $\hat{X}(t)$ for the
projection  $X(t)$ onto the hyperplane
${\mathbf R}^N_0=\{x \in {\mathbf R}^N: \sum x_i=0\}$.  
Suppose that $X$
when governed by a  probability
measure ${\mathbf P}^{N,\theta}_x$ evolves as the family of  $N$ mutually  sticky Brownian motions 
associated with a parameters  $\theta=(\theta(k:l); k,l \geq 1)$
started from $x \in {\mathbf R}^N$.
Consider, for $\epsilon>0$,  the neighbourhood $D(\epsilon)$ of the origin $0$ in
${\mathbf R}^N_0$ given by 
\begin{equation}
D(\epsilon)=\{x \in {\mathbf R}^N_0: \max_{i,j} (x_i-x_j) \leq \epsilon\}.
\end{equation}
We know from \cite{hw1} that the exit distribution of $\hat{X}$ from 
$D(\epsilon)$ can, for small $\epsilon$,  be described in terms of the
$\theta(k:l)$ parameters. In fact if $T(\epsilon)$ denotes  the first time
that $\hat{X}$ leaves this set, we have
\begin{equation}
\label{theta1}
\lim_{\epsilon \downarrow 0} \frac{1}{\epsilon}{\mathbf E}^{N,\theta}_0 \bigl[ T(\epsilon) \bigr]=
\frac{1}{ 2\sum_{k=1}^{N-1} { N \choose k}  \theta(k:N-k) },
\end{equation}
and, for each cell $E$ that corresponds to a (ordered) partition of
$\{1,.2, \ldots, N\}$ into two parts having sizes $k$ and $l=N-k$,
\begin{equation}
\label{theta2}
\lim_{\epsilon \downarrow 0} {\mathbf P}^{N,\theta}_0 \bigl(
X(T(\epsilon)) \in E 
)= 
\frac{\theta(k:l)}{ \sum_{k=1}^{N-1} { N \choose k}  \theta(k:N-k)}
\end{equation}
Notice how this is consistent with the  idea that
$\theta(k,N-k)$ describes the rate at which a cluster of  $N$ particles
splits.

In view of these observations on the behaviour of sticky diffusions we
can reasonably expect to be able to identify the parameters
$\theta(k:l) $ arising in the limiting behaviour of our $N$ point
motions with generators \eqref{pertgen} by investigating how these processes,
for $n$ large,
leave neighbourhoods of the  origin. Interestingly very close to the
origin, at distances of the order $1/n^2$, the $N$ point motions  are
spherically symmetric, but at larger distances a coalescence effect
leads to exit distributions  concentrated  on  points
corresponding to the cluster of particles  splitting into two subclusters.  

We will suppose that $X$ when governed by probability measures
${\mathbf P}^{N,n}_x$  evolves
as a   diffusion with generator  ${\mathcal G}^{N,n}$
starting from $x\in {\mathbf R}^N$.  Notice that the generators ${\mathcal G}^{N,n}$ are invariant under shifts
$(x_1,x_2, \ldots x_N) \mapsto (x_1+h,x_2+h, \ldots x_N+h)$, and
consequently the projection $\hat{X}(t)$ of $X(t)$  is a diffusion
also. In view of \eqref{theta1} and \eqref{theta2} it is natural to study the exit
time and distribution of $\hat{X}$ from  $D(\epsilon)$ under ${\mathbf
  P}^{N,n}_0$ in order to determine the parameters $\theta(k:l)$
associated with the limiting $N$ point motion.  We will  estimate the
exit distribution (non-rigorously) by
approximating the behaviour of $\hat{X}$  on two different scales.

Let $B(r)$ denote  the ball of radius $r$ in ${\mathbf R}^N_0$, 
\begin{equation*}
B(r)= \{ x \in {\mathbf R}^N_0: \| x\| \leq r\}.
\end{equation*}
Now, for a fixed  small $\epsilon>0$,  the map $ x \mapsto \psi(x)$ is approximately quadratic for
$x \in (-\epsilon, \epsilon)$ and we use this to approximate  the covariance matrix of $\hat{X}$ in the ball $B(\epsilon/n)$. Observe that if the  matrix $A$ has entries $1-a^2(x_i-x_j)^2$ then for vectors $u,v \in {\mathbf R}^N_0$ we have $(u,Av)=2a^2(u,x)(v,x)$.
Consequently we can approximate $\hat{X}$ under ${\mathbf P}^{N,n}$
within the ball $B(\epsilon/n)$ as  $(n^{-2} Z(n^2t ); t \geq 0 )$ where $Z$ is a diffusion
with generator ${\mathcal H}^{N}$ given by, 
in spherical co-ordinates in ${\mathbf R}^{N}_0$,
\begin{equation}
\label{lineargen}
{\mathcal H}^N= a^2 r^2\frac{\partial^2}{\partial r^2}  +\frac{b^2}{2}
\nabla^2 = \left(\frac{b^2}{2}+a^2r^2 \right)\frac{\partial^2}{\partial r^2}+ \frac{(N-2)b^2}{2r}\frac{\partial}{\partial r} +\frac{b^2}{2 r^{2}}\Delta_{S^{N-2}}. 
\end{equation}
In particular,  the rescaled radial part of $\hat{X}$  
is   approximated as a 
diffusion on $(0,\infty)$ with  generator 
\begin{equation}
\label{radialgen}
{\mathcal H}^N_{\text{rad}}=\left(\frac{b^2}{2}+a^2r^2\right) \frac{d^2}{dr^2}+\frac{(N-2)b^2}{2r} \frac{d}{dr}.
\end{equation}
The expected  time taken for this diffusion to first reach a  level $r$  when started from $0$  is equal to  $f_0(r)$ where  $f_0$ is the increasing solution to
\[
{\mathcal H}^N_{\text{rad}} f_0=1,  \qquad\qquad  f_0(0)=0.
\]
 The function $f_0(r)$  is asymptotically equal to $ r/ (\gamma ab)$,  see \cite{w2}, where
\begin{equation}
\label{gamma}
\gamma=\sqrt{\frac{2}{ \pi}}\frac{\Gamma(N/2)}{\Gamma((N-1)/2)}= \frac{1}{\sqrt{\pi}}\int_{{\mathbf R}^{N-1}}  \frac{\|x\|e^{-\|x\|^2/2}}{(2\pi)^{(N-1)/2}} \; dx.
\end{equation} 
 Thus we have the estimate 
\begin{equation}
\label{timeball}
{\mathbf E}_0^{N,n}[ \text {exit time from $B(\epsilon/n)$ }] \approx  \frac{\epsilon}{n  \gamma ab}.
\end{equation}
 Moreover, because of the spherical symmetry of  ${\mathcal H}^N$, the exit
distribution from this ball is the uniform measure on sphere. 

We next consider $\hat{X}$ started from a  point $x$ on
the sphere of radius $\epsilon/n$ which we will assume has distinct
co-ordinates. Let $\sigma$ be the permutation so that
$x_{\sigma(1)}  >x_{\sigma(2)}> \cdots> x_{\sigma(N)} $, and denote by
$x^\sigma$ the vector $(x_{\sigma(1)},x_{\sigma(2)}, \cdots,  x_{\sigma(N)})$ .    Our second approximation  applies to $\hat{X}$ until it  first leaves the domain
$D(\epsilon) \setminus D( 1/(\epsilon n^2))$. If two particles  come close to each other, then they have a negligible probability of separating by a significant distance prior to the exit time $\tau$ from the domain. Thus we can treat 
 $\hat{X}$  similarly to  (the projection to ${\mathbf R}^N_0$) of a system of $N$ coalescing Brownian motions. In particular this means that if $\hat{X}$  exits via the outer part of the boundary then it does so with $\hat{X}^\sigma_{1}(\tau)-\hat{X}^\sigma_{N}(\tau) \approx \epsilon$. Consequently applying the optional stopping Theorem to the martingale $\hat{X}^\sigma_{1}(t)-\hat{X}^\sigma_{N}(t)$  gives rise to the estimate
\begin{equation}
{\mathbf P}^{N,n}_x \bigl( \hat{X} \text{ exits $D(\epsilon) \setminus
    D( 1/(\epsilon n^2))$ via the outer boundary }  \bigr)\approx   \frac{x^\sigma_{1}-x^\sigma_{N}}{\epsilon}.
\end{equation}
Moreover if $\hat{X}$ does exit via the outer boundary then as it does so there are only two clusters of particles (see Lemma \ref{coalesce} for the corresponding statement about coalescing Brownian motion), and applying the optional stopping Theorem to  $\hat{X}^\sigma_{k}(t)-\hat{X}^\sigma_{k+1}(t)$ gives
\begin{equation}
\label{splitdist}
{\mathbf P}^{N,n}_x \bigl( \hat{X}^\sigma_{i}(\tau) -
\hat{X}^\sigma_{i+1}(\tau) \approx 0 \text{ for $i \neq k$,  }  \hat{X}^\sigma_{k}(\tau) -
\hat{X}^\sigma_{k+1}(\tau) \approx \epsilon
  \bigr) \approx  \frac{x^\sigma_{k}-x^\sigma_{k+1}}{\epsilon}.
\end{equation}

We now make use of a  renewal argument. 
The diffusion with generator \eqref{lineargen} is ergodic, and
consequently we conclude that the process $\hat{X}$ spends all but a
negligible amount of time at  a distance of order $1/n^2$ from the
origin  prior to exiting $D(\epsilon)$. From this inner region it makes 
excursions to the sphere of radius $\epsilon/n$ and, each time it
does, it has a small probability of exiting $D(\epsilon)$ rather than
returning to the inner region. When it does return to distances of
order $1/n^2$ we can assume by mixing that it is starts afresh and forgets its history. Thus
$\hat{X}$ makes approximately a geometrically distributed number of
excursions  to the sphere of radius $\epsilon/n$ before exiting
$D(\epsilon)$, and we
conclude, neglecting  the time spent outside the ball $B(\epsilon/n)$, that the expected time to exit $D(\epsilon)$ is estimated by 
\begin{equation}
\frac{{\mathbf E}_0^{N,n}[ T_{B(\epsilon/n)}] }{
 {\mathbf E}_0^{N,n} \bigl[ {\mathbf P}^{N,n}_{X (T_{B(\epsilon/n)})} \bigl( \hat{X} \text{ exits $D(\epsilon) \setminus
    D( 1/(\epsilon n^2))$ via the outer boundary }  \bigr)\bigr]},
\end{equation}
where $T_{B(\epsilon/n)}$ denotes the first time of exiting the ball $B(\epsilon/n)$.
Similarly we estimate that the probability of exiting $D(\epsilon)$ at time $T_{D(\epsilon)}$
with  $\hat{X}_{i+1}(T_{D(\epsilon)})- \hat{X}_i  (T_{D(\epsilon)}) \approx 0 $
for all $i \neq k$ and  $\hat{X}_{k+1}(T_{D(\epsilon)})- \hat{X}_k  (T_{D(\epsilon)}) \approx \epsilon $ is approximately
\begin{equation}
\frac{ {\mathbf E}_0^{N,n} \bigl[ {\mathbf P}^{N,n}_{X (T_{B(\epsilon/n)})} \bigl( \hat{X}_{i}(\tau) -
\hat{X}_{i+1}(\tau) \approx 0 \text{ for $i \neq k$,  }  \hat{X}_{k}(\tau) -
\hat{X}_{k+1}(\tau) \approx \epsilon  \bigr)\bigr]} {{\mathbf E}_0^{N,n} \bigl[ {\mathbf P}^{N,n}_{X (T_{B(\epsilon/n)})} \bigl( \hat{X} \text{ exits $D(\epsilon) \setminus
    D( 1/(\epsilon n^2))$ via the outer boundary }  \bigr)\bigr]},
\end{equation}
where, as previously, $\tau$ is the exit time of $D(\epsilon) \setminus
    D( 1/(\epsilon n^2))$.
Thus, in view of \eqref{theta1} and \eqref{theta2}, and   taking the cell $E=\{ x_1= x_2= \ldots= x_k< x_{k+1}=x_{k+2}= \ldots =x_N\}$, we guess  that the parameter $\theta(k:N-k)$ 
associated with a limiting $N$-point motion should be equal to the
limit as $n$ tends to infinity and $\epsilon$ tends to zero of
\begin{equation}
\frac{\epsilon \times {\mathbf E}_0^{N,n} \bigl[ {\mathbf P}^{N,n}_{X (T_{B(\epsilon/n)})} \bigl( \hat{X}_{i}(\tau) -
\hat{X}_{i+1}(\tau) \approx 0 \text{ for $i \neq k$,  }  \hat{X}_{k}(\tau) -
\hat{X}_{k+1}(\tau) \approx \epsilon  \bigr)\bigr] }{ 2 {\mathbf E}_0^{N,n}[ T_{B(\epsilon/n)} ] }
\end{equation}
Substituting in our estimates from \eqref{timeball} and \eqref{splitdist} and using the fact  that
the exit distribution  from $B(\epsilon/n)$ is uniform we arrive at
\begin{equation}
 \frac{\gamma ab}{2}
\int_{S^{N-2}}  \bigl(\min_{1 \leq i \leq k} z_i -\max_{k+1 \leq i
  \leq N}z_i )\bigr)^+dz .
\end{equation}
in which the integral over the unit sphere $S^{N-2}\subset {\mathbf R}^N_0$ is taken with
respect to  Lebesgue measure on the sphere normalized so $\int_{S^{N-2}} dz=1$.  When we rewrite the spherical
integral as a Gaussian integral this agrees  the value given in Theorem \ref{main}.

\section{Proof of main result}

In view of the characterization of a family of sticky Brownian
motions by the ${\mathcal A}^\theta_N$-martingale problem, it is a
natural strategy to prove Theorem \ref{main} by considering smooth
approximations $f_n$ to a given function $f\in L_N$ and to derive, using weak
convergence,  from the
martingale property, under ${\mathbf P}^{N,n}$, of 
\begin{equation}
\label{gmart}
f_n(X(t))- \int_0^t {\mathcal G}^{N,n} f_n (X(s))ds
\end{equation}
that
\[
f(X(t))- \int_0^t {\mathcal A}^{\theta}_N f (X(s))ds,
\]
is a martingale under ${\mathbf P}^{N,\theta}$.
There are  difficulties to be overcome in pursuing this which
 arise because ${\mathcal A}^{\theta}_N f$  is not continuous. A key
 step is to establish the  weaker statement described in the following
 lemma, which gives information about how the limiting process leaves
 the main diagonal $D=\{ x \in {\mathbf R}^N: x_1=x_2= \ldots= x_N\}$. Let $L_N^0$ denote the subspace of  $L_N$ containing those functions which are invariant under shifts $(x_1,x_2, \ldots, x_n) \mapsto (x_1+h,x_2+h, \ldots, x_n+h)$,  and consequently identically equal  to $0$ on $D$.

\begin{lemma}
\label{essentiallemma}
Fix $x\in {\mathbf R}^N$, and suppose that ${\mathbf P}_{x}$ is a
subsequencial limit of the family of probability measures $\bigl({\mathbf P}_{x}^{N,n} ; n \geq
1\bigr)$. Then for any convex $f \in L^N_0$, 
  \[
Z^f(t)=f(X(t))- {\mathcal A}^\theta_Nf(0) \int_0^t {\mathbf 1}(X(s) \in D) ds 
\]
is a submartingale under ${\mathbf P}_x$, where the  family of
parameters $\theta$ are specified as in Theorem 1.
\end{lemma}

We will  prove this lemma  by applying  weak convergence to
${\mathbf P}^{N,n}$ martingales given at \eqref{gmart}. But it turns out
that we must carefully  select  suitable smooth approximations
$f_n$. In fact we will choose $f_n(x)= n^2 g(n^{-2}x)$ where the
function $g$ is determined according to the next proposition which is adapted  from \cite{w2}.

 Recall that   the
generators ${\mathcal G}^{N,n}$, rescaled and restricted to ${\mathbf R}^{N}_0$,  converge to ${\mathcal H}^{N}$ given  by\eqref{lineargen}. The constant $gamma$ was defined at \eqref{gamma}.

\begin{proposition}
\label{prop}
Let $f: S^{N-2} \rightarrow {\mathbf R}$ be a square integral function on the
unit sphere  $S^{N-2} \subset {\mathbf R}^{N-1}_0$. Let 
\[
c=c(f)=  \gamma ab  \int_{S^{N-2}} f(z) dz
\]
 where the
integral is with respect to normalized  Lebesgue measure on the sphere.   There exists a  unique solution to
\[
{\mathcal H}^N g=  c
\]
satisfying $g(0)=0$ and 
\[
\lim_{r \rightarrow \infty} g(rz)/r= f(z) \text{ uniformly for } z \in
S^{N-2}.
\]
Moreover if $ y \mapsto\|y\| f(y/\|y\|)$ is a convex function  on
${\mathbf R}^{N-1}_0$ then so too is $y \mapsto g(y)$.
\end{proposition}

\begin{proof}[proof of Lemma \ref{essentiallemma}]
Let $f \in L^N_0$ be convex, and consider its restriction to
$S^{N-2} \subseteq {\mathbf R}^N_0$. Let $c=c(f)=\gamma ab \int_{S^{N-2}} f(z)dz$ and let $g$ be the
corresponding solution to  $ {\mathcal H}^N g = c$  described in
Proposition \ref{prop}. Extend $g$ to a function on ${\mathbf R}^N$ invariant
under shifts $(x_1,x_2, \ldots, x_n) \mapsto (x_1+h,x_2+h, \ldots, x_n+h)$, and set $g_n(x)=n^{-2} g(n^2x)$.

We  want to estimate  ${\mathcal G}^{N,n} g_n(x)$ in a
neighbourhood of the diagonal $D$.  We write
\begin{multline} \label{decomp}
{\mathcal G}^{N,n}g_n(x)= \frac{1}{2}\sum_{i, j}\psi \bigl(n(x_i-x_j)\bigr)\frac{\partial^2
}{\partial x_ix_j} g_n (x)+ \frac{b^2}{2n^2} \sum_{i}\frac{\partial^2
}{\partial x_i^2}g_n(x) \\
= \left\{\frac{1}{2}\sum_{i, j}\bigl(\psi\bigl(n(x_i-x_j)\bigr)-1+a^2n^2(x_i-x_j)^2\bigr)\frac{\partial^2
}{\partial x_ix_j} g_n (x) \right\} + \\
\left\{\frac{1}{2}\sum_{i, j}\bigl(1-a^2n^2(x_i-x_j)^2\bigr)\frac{\partial^2
}{\partial x_ix_j} g_n (x)+ \frac{b^2}{2n^2} \sum_{i}\frac{\partial^2
}{\partial x_i^2}g_n(x) \right\}
\end{multline}
The first term in braces appearing here can be controlled as
follows. Recall  $\hat{x}$ denotes the  orthogonal projection of $x$ onto ${\mathbf R}^N_0$ and that $B(r)$ is the ball of radius $r$ in ${\mathbf R}_0^N$.  Given  $K>0$ let 
\[
M(K)=\max_{i,j}   \sup_{\hat{x} \in B(K)}\left|\frac{\partial^2
}{\partial x_ix_j} g (x) \right| = n^{-2}  \max_{i,j}   \sup_{\hat{x} \in
B(K/n^2)}\left|\frac{\partial^2}{\partial x_ix_j} g_n (x) \right| < \infty.
\]
Then  given $\epsilon>0$, we may by \eqref{ass}, choose $n_0$  so that
for all $n \geq n_0$, and $x$ so that $\hat{x} \in B(K/n^2)$,
\[
\bigl| \psi_n( x_i-x_j) -1 +a^2n^2( x_i-x_j)^2 \bigr|  \leq
\frac{\epsilon }{N^2 KM(K)}n^2(x_i-x_j)^2 \leq \frac{\epsilon }{N^2 M(K)}, 
\]
and this then entails that the first term in braces is no larger than
$\epsilon$ in modulus. Because of the shift invariance of $g$, the second term in braces appearing in equation  \eqref{decomp} is  equal to $\bigl({\mathcal H}^N g\bigr) ( n^2 x)$, which in
turn is equal to $c(f)$. 

Next  we claim that
\[
c(f) ={\mathcal A}^\theta_Nf(0).
\]
To verify this it is enough, by linearity, to check it for functions of
the form 
\[
f(x)= \bigl(\min_{i \in \pi_1} x_i - \max_{i \in \pi_2} x_i \bigr)^+
\]
where $\pi=(\pi_1,\pi_2)$ is an ordered partition of $\{1, \ldots,N\}$ into two non-empty parts.
For such $f$ the gradients  $\nabla_{v} f(0)$ appearing in the
definition of ${\mathcal A}^\theta_N f(0)$ are all zero except for $\nabla
v_{\pi_1, \pi_2} f(0)$ which equals $2$. Thus, recalling the values  assigned to the parameters $(\theta(k:l)$ in Theorem 1,
\[
{\mathcal A}^\theta_N f(0)=2 \theta(|\pi_1|, |\pi_2|) \\
= \frac{ab}{\gamma_N}
\int_{S^{N-2}}  \bigl(\min_{i \in \pi_1} z_i -\max_{i \in \pi_2}z_i
)\bigr)^+dz
=c(f).
\]

Observe that because $g_n$ is smooth and convex,  ${\mathcal G}^{N,n} g_n$ is continuous and  non-negative everywhere.
This fact, together with the above paragraphs allows us to conclude
that given   $K>0$ and  $\epsilon>0$,   for all
sufficiently large $n$ we have 
\begin{equation}
\label{subby}
g_n(X(t))- \bigl({\mathcal A}^\theta_Nf(0)- \epsilon \bigr) \int_0^t {\mathbf 1}(\hat{X}(s) \in B(K/n^2)) ds
\end{equation}
is a submartingale under ${\mathbf P}^{N,n}$.  

Fix times $ s<t$ and let $\Phi$ be a bounded, non-negative and continuous function on the path space ${\mathbf C}\bigl( [0,s],{\mathbf R}^N\bigr)$ .  
Note that the boundary behaviour of  $g$ implies that $| g_n(x) -f(x)|/ (1+||x||) \rightarrow 0 $ as $n \rightarrow \infty$ uniformly for $x \in {\mathbf R}^N$ , and that  since 
${\mathbf E}^{N,n} \bigl[ \|X(s)\|\bigr] $ and ${\mathbf E}^{N,n} \bigl[ \|X(t)\|\bigr] $ are bounded uniformly in $n$,
the weak convergence of 
( a subsequence of ) ${\mathbf P}^{N,n}$ to ${\mathbf P}$, implies that ( along the subsequence)
\begin{multline*}
{\mathbf E}^{N,n}\Bigl[ \Phi( X(r), r \leq s)\bigl( g_n( X(t))- g_n(X(s)) \bigr) \Bigr]   \rightarrow \\
{\mathbf E} \Bigl[ \Phi( X(r), r \leq s)\bigl( f( X(t))- f(X(s)) \bigr) \Bigr].
\end{multline*}
Let $\phi_K:{\mathbf R}^N_0 \rightarrow [0,1]$ be a  continuous  function satisfying $\phi_K(x)=0$ for $\|x\|\geq 1/K$ and $\phi_K(x)=1$ for $\|x\|\leq 1/(2K)$
Then we also have by weak convergence ( along the subsequence) that
 \begin{multline*}
{\mathbf E}^{N,n}\left[ \Phi( X(r), r \leq s)\int_s^t \phi_K\bigl(\hat{X}(u)\bigr) du    \right]   \rightarrow {\mathbf E}\left[ \Phi( X(r), r \leq s)\int_s^t \phi_K\bigl(\hat{X}(u)\bigr) du    \right]   \\ \geq 
{\mathbf E}\left[ \Phi( X(r), r \leq s)\int_s^t {\mathbf 1}(X(u) \in D) du    \right] .
\end{multline*}
For a given $\epsilon>0$, if we choose $K$ large enough, then  by virtue of Lemma  \ref{nearzero},   for all sufficiently large $n$,
\begin{multline*}
{\mathbf E}^{N,n}\left[ \Phi( X(r), r \leq s)\int_s^t {\mathbf 1}(\hat{X}(u) \in B(K/n^2)) du    \right]+\epsilon \geq \\
 {\mathbf E}^{N,n}\left[ \Phi( X(r), r \leq s)\int_s^t \phi_K\bigl(\hat{X}(u)\bigr) du    \right] .
\end{multline*}

From these    statements  and the fact that  the process at \eqref{subby} is a submartingale for large enough $n$, it follows that
\begin{multline*}
{\mathbf E} \Bigl[ \Phi( X(r), r \leq s)\bigl( f( X(t))- f(X(s)) \bigr) \Bigr] \geq 
\\
  \bigl({\mathcal A}^\theta_Nf(0)- \epsilon \bigr)\left({\mathbf E}\left[ \Phi( X(r), r \leq s)\int_s^t {\mathbf 1}(X(u) \in D) du    \right] -\epsilon \right)
\end{multline*}
Consequently,  $s \leq t$, $ \Phi\geq 0$ and $\epsilon>0$ being arbitrary,  $Z^f$ is a submartingale under ${\mathbf P}$ as desired.
\end{proof}

We may now give the
\begin{proof}[Proof of Theorem 1]
Fix $x_0\in {\mathbf R}^N$.  Because the marginal  laws of each component $\bigl( X_i(t); t \geq 0)$ converge as $n \rightarrow \infty$ it follows that  the family of probability measures $\bigl({\mathbf P}_{x_0}^{N,n} ; n \geq
1\bigr)$ is tight. Thus it suffices to show that any limit point
${\mathbf P}_{x_0}$ solves the ${\mathcal A}_N^\theta$-martingale problem
starting from $x_0$. 

We know that each pair of components $(X_i,X_j)$ converges in law to a pair of
 Brownian motions whose difference is a sticky Brownian motion and consequently 
\[
2\langle X_i, X_j \rangle (t)= \langle X_i, X_i\rangle (t) +\langle X_j, X_j \rangle (t) - \langle X_i-X_j \rangle (t)  =2\int_0^t {\mathbf 1}(X_i(s) \neq X_j(s))ds
\]
under ${\mathbf P}_{x_0}$. Thus it suffices to  show that 
\begin{equation}\label{mart}
f(X(t))- \int_0^t {\mathcal A}^\theta_N f(X(s)) ds 
\end{equation}
is a ${\mathbf P}_{x_0}$-martingale for each $f \in L^N$. By
the addition of a suitable linear function we may assume that $f \in
L^N_0$. In fact  we claim  that
it is enough  that for every convex  $f\in L^N_0$  the expression
at \eqref{mart} defines a submartingale. We verify this claim as follows.
For a general $f$
we may consider $g(x)= c\sum_{i<j} |x_i-x_j| + f(x)$  which for sufficiently
large $c$ is convex.  We would then have  that the corresponding
process  $g(X(t))- \int_0^t {\mathcal A}^\theta_N g(X(s)) ds$  is a
submartingale.   But we also know  that the difference of  each pair of components of $X$ is a sticky Brownian motion  with parameter $\theta=2\theta(1:1)$, and  thus,  
\[
|X_i(t)- X_j(t)|- 4\theta(1:1)\int_0^t {\mathbf 1}(X_i(s) = X_j(s))ds
\]
is a martingale.  Now we also observe that 
\[
{\mathcal  A}_N^\theta g (x)= 4c\theta(1:1) \sum_{i<j} {\mathbf 1}(x_i=x_j) +{\mathcal
  A}_N^\theta f(x).
\]
And so  we  deduce that  \eqref{mart} must be a
submartingale. But we can  consider $g(x)= c\sum |x_i-x_j| -
f(x)$ in the same manner, and hence  deduce that \eqref{mart} is a supermartingale.

We now proceed with the proof of the theorem. The result holds for
dimension $N=2$, and we argue by induction on $N$. So assume  the result
holds for dimension $N-1$, and consider a convex $f \in L^N_0$. 
By the Meyer decomposition theorem, associated with the ${\mathbf P}_{x_0}$
submartingale $f(X(t))$ is some continuous  increasing process $A(t)$.   
Let $U_{\pi} = \{ x \in {\mathbf R}^N: x_i > x_j \text {for all  }i \in
  \pi_1, j \in \pi_2 \}$ for some ordered partition $\pi=(\pi_1,\pi_2)$
  of   $\{1,2,\ldots,N\}$ into two parts.
According to Lemma \ref{local}, on $U_\pi$, $f(x)$ can be written as a
sum  of $f_1(x_j ;j \in \pi_1)$ and $f_2(x_j; j \in \pi_2)$. Applying the
inductive hypothesis  the processes
\[
f_i(X_j(t); j \in \pi_i ) - \int_0^t {\mathcal A}_{\pi_i}^\theta f_i
(X_j(s); j \in \pi_i) ds
\]
 for $i=1,2$ are both
martingales. Consequently,  the
compensator $A$  of $f(X(t))$ must satisfy 
\[
 dA(t)=  \bigr({\mathcal A}_{\pi_1}^\theta f_1
(X_j(t); j \in \pi_1) + {\mathcal A}_{\pi_2}^\theta f_i
(X_j(t); j \in \pi_2)  \bigl) dt
\]
on the set $\bigl\{ t: X(t) \in U_{\pi}\}$.
Noting  that \[
 \bigr({\mathcal A}_{\pi_1}^\theta f_1
(x_j; j \in \pi_1) + {\mathcal A}_{\pi_2}^\theta f_i
(x_j; j \in \pi_2)  \bigl)= {\mathcal A}_N^\theta f (x) \text { for }
x \in U_\pi, 
\] and  letting $\pi$ vary we conclude that in fact
\[
 dA(t)= {\mathcal A}_N^\theta f (X(t)) dt \text{ on } \{ t: \hat{X}(t)
 \neq 0\}.
\]
Finally applying Lemma \ref{essentiallemma} we deduce that $dA$ must
dominate ${\mathcal A}_N^\theta f (X(t)) dt$ on $ \{ t: \hat{X}(t)
 \neq 0\} $ and that \eqref{mart} must be a submartingale. By our  previous
discussion since this holds for every convex f$\in L^N_0$ in fact
\eqref{mart} is a martingale and  the inductive step is complete.

\end{proof}

\section{Some lemmas}
\begin{lemma}
\label{nearzero}
Given     $t$ and $\epsilon>0$ there exist $c$,$c^\prime$ and $n_0$    such that
\[
{\mathbf E}_x^{N,n} \left[ \int_0^t {\mathbf 1} \bigl( |X_i(s)-X_j(s)| \in (
  c/ n^2, c^\prime)\bigr) ds \right]  \leq \epsilon
\]
for all $n\geq n_0$ and $x\in {\mathbf R}^N$. 
\end{lemma}
\begin{proof}
Under ${\mathbf P}_x^{N,n}$, the process $Z=X_i-X_j$ is a diffusion in natural scale with
speed measure $m_n$ given by \eqref{speed}. It can thus  be  represented  as a time changed Brownian motion:
\[
Z(t)=B(\tau^n_t),
\]
where $\tau^n$ is the inverse of the increasing functional
\[
\frac{1}{2}\int_0^u \frac{ds}{1+b^2n^{-2}-\psi(nB(s))}
\]
and $B$ a standard Brownian motion starting from $x_i-x_j$.
Consequently 
\[
\int_0^{\tau_t^n} {\mathbf 1}_{( c/n^2, c^\prime)} \bigl(| B(s)| )\bigr) \frac{ds}{1+b^2n^{-2}-\psi(nB(s))}
\]
is a random variable  with the same distribution as 
2\[\int_0^t {\mathbf 1} \bigl( |X_i(s)-X_j(s)| \in ( c/n^2), c^\prime)\bigr) ds\]
has  under ${\mathbf P}_x^{N,n}$.
Note that  for all  sufficiently large $n$,
\[ \frac{1}{2} \leq \frac{1}{1+b^2n^{-2}-\psi(nz))}  \text { for all } z \in
  {\mathbf R}
\]
whence $\tau_t^n \leq 4 t$ and 
\[
\int_0^{\tau_t^n} {\mathbf 1}_{ ( c/n^2, c^\prime)}\bigl(| B(s)|  \bigr) \frac{ds}{1+b^2n^{-2}-\psi(nB(s))}
\leq  \int_0^{4t}  f_n(B(s))ds
\] 
where $f_n(z)= {\mathbf 1}_{( c/n^2, c^\prime)}(|z|) \bigl(1+b^2n^{-2}-\psi(nz)\bigr)^{-1}$.
Now rewriting this integral using the occupation time formula, and taking expectations we see that it is enough to verify that
\[
\int_{\mathbf R} f_n(z) dz 
\]
can be made arbitrarily small for all sufficiently large $n$    $c$ sufficiently large 
and $c^\prime$ sufficiently small.  This is easily checked using the assumptions on $\psi$ and in   particular 
using that there is a $\delta>0$ and a constant $M<\infty$   so that for all
sufficiently large $n$,
\[
  \bigl(1+b^2n^{-2}-\psi(nz)\bigr)^{-1} \leq \frac{2n^2}{2b^2+a^2n^4z^2},  \text{ for } z \in (-\delta/n,\delta/n)
\]
whilst  
\[
\bigl(1+b^2n^{-2}-\psi(nz)\bigr)^{-1} \leq M \text{ for } z \in {\mathbf R}\setminus (-\delta/n,\delta/n).
\]

\end{proof}

\begin{lemma}
\label{local}
Let $\pi=(\pi_1, \pi_2)$ be an ordered partition of $\{1,2, \ldots , N\}$
  into two non-empty parts, and define
\[
U_{\pi} = \{ x \in {\mathbf R}^N: x_i > x_j \text {for all  }i \in
  \pi_1, j \in \pi_2 \}.
\]
Then  $f \in L^N$ can be expressed as
\[
f(x)= f_1( x_i ; i \in \pi_1) + f_2(x_j; j \in \pi_2) \text{ for all }
x \in U_{\pi}
\]
for  some $f_1 \in L^{|\pi_1|}$, $f_2 \in L^{|\pi_2|}$.
\end{lemma}
\begin{proof}
By subtracting a linear function we can assume $f\in L^N_0$.  Now suppose that a given $x\in U_\pi$ satisfies $x_i >0> x_j$ for all  $i \in \pi_1, j \in \pi_2$.  Let $ y\in {\mathbf R}^N$ have  components $y_i=x_i$ for $i \in \pi_1$ and $ y_i=0$ otherwise. Likewise let $ z\in {\mathbf R}^N$ have  components $z_i=x_i$ for $i \in \pi_2$ and $ z_i=0$. Then both $y$ and $z$ lie in the closure of the cell that contains $x$, and by the linearity of $f$ restricted to the closure of that cell,
\[
f(x)=f(y)+f(z).
\]
Consequently we define $f_1( x_i ; i \in \pi_1)= f(y)$ and $f_2(x_j; j \in \pi_2)=f_2(z)$, extending each linearly within cells so as to functions  $f_1 \in L^{|\pi_1|}$ and  $f_2 \in L^{|\pi_2|}$.

\end{proof}

\begin{lemma}
\label{coalesce}
Suppose that $B_1(t) \geq B_2(t)\geq \cdots \geq B_N(t)$ are a system of coalescing Brownian motions on ${\mathbf R}$. Let $T_R= \inf\{ t\geq 0: B_1(t)-B_N(t)=R\}$, and let $r$ denote  $ B_1(0)-B_N(0)$. Then there exists a constant $C$ such that for all  $r$ and $R$ with $0 \leq r \leq R/2$,
\begin{equation*}
{\mathbf P}\bigl( T_R <\infty \text{ and there exists some $i$ with } B_1(T_R)>B_i(T_R)> B_N(T_R) \bigr) \leq C (r/R)^3.
\end{equation*}
\end{lemma}
\begin{proof}
For $i=2,3, \ldots, N-1$, let $A_i$ be the event 
\[
T_R <\infty \text{ and  } B_1(T_R)>B_i(T_R)> B_N(T_R)
\]
Since the event in question is the union of these events, it is enough to prove the desired estimate holds for each $A_i$. Projecting the three dimensional process $ \bigl(B_1(t),B_i(t), B_N(t)\bigr)$ onto the plane $\{ x \in {\mathbf R}^3: x_1+x_2+x_3=0\}$ we see $A_i$ can be identified with the event that a two dimensional Brownian motion  started at a point satisfying $y_1=r$ exits the domain 
\begin{equation*}
\{ y \in {\mathbf R}^2: 0\leq y_1 \leq R,   | y_2| \leq y_1/\sqrt{3}  \}
\end{equation*}
via the boundary  $y_1= R$. By comparing with a wedge with a circular outer boundary and interior angle $\pi/3$ and solving the appropriate Dirichlet problem this exit probability is  easily seen to by bounded by  $C(r/R)^3$.
\end{proof}

\section{Stochastic flows of kernels}

Returning to the motivation coming from Gaw\c{e}dzki and Horvai it is natural to interpret the results from this paper in terms of the stochastic flows. As remarked in the introduction the consistent family of $N$ point motions with generators ${\mathcal G}^{N,n}$ do not correspond to any stochastic flow of maps. However according to  the  theory developed by Le Jan and Raimomd \cite{lejan2} they are associated with the  more general notion of a flow of kernels.

Let $W= \bigl( W(t,x), t \geq 0, x \in {\mathbf R} \bigr)$ denote the centred Gaussian process   with  covariance function $\psi(n(x_1-x_2)) \min( t_1, t_2)$. Suppose $B_1, B_2, \ldots, B_N$ are real valued Brownian motions, independent of each other and $W$. Then a diffusion with generator ${\mathcal G}^{N,n}$ can be obtained, at least in a formal sense, by solving the stochastic differential equations
\begin{equation}
X_i(t)= x_i+ \int_0^t dW(s, X_i(s))+ \frac{\sigma}{n} B_i(t).
\end{equation}
The stochastic flow of kernels $\bigl(K_{s,t}, s \leq t\bigr)$ associated with family   ${\mathcal G}^{N,n}$ describes a cloud of infinitesimal particles moving in this manner. It can be obtained by filtering on $W$, 
\begin{equation}
K_{0,t}(x_1,A)= {\mathbf P} \bigl( X_1(t) \in A | W \bigr).
\end{equation}
These kernels  have smooth densities which satisfy a  stochastic partial differential equation of advection-diffusion type. If $v(t,y)$ denotes the density of  $\int v(0,x) K_{0,t}(x,\cdot) dx$ at $y$, then
\begin{multline}
v(t,y)-v(0,y)=\int_0^t  \frac{\partial v }{\partial y} (s,y)  dW(s,y) + \int_0^t  v(s,y)  dW_y(s,y) + \\
\frac{1}{2}( b^2+1)\int_0^t  \frac{\partial^2 v }{\partial y^2} (s,y)  ds,
\end{multline}
where $W_y(t,y)= \partial W(t,y)/\partial y $. Simulations showing  a realization of  the density of $K_{0,1}(0,\cdot)$ for two different sets of  parameter values  are shown in Figure \ref{fig}. 

\begin{figure}
\centering
\begin{minipage}{.45\textwidth}
  \centering
    \includegraphics[width=\textwidth]{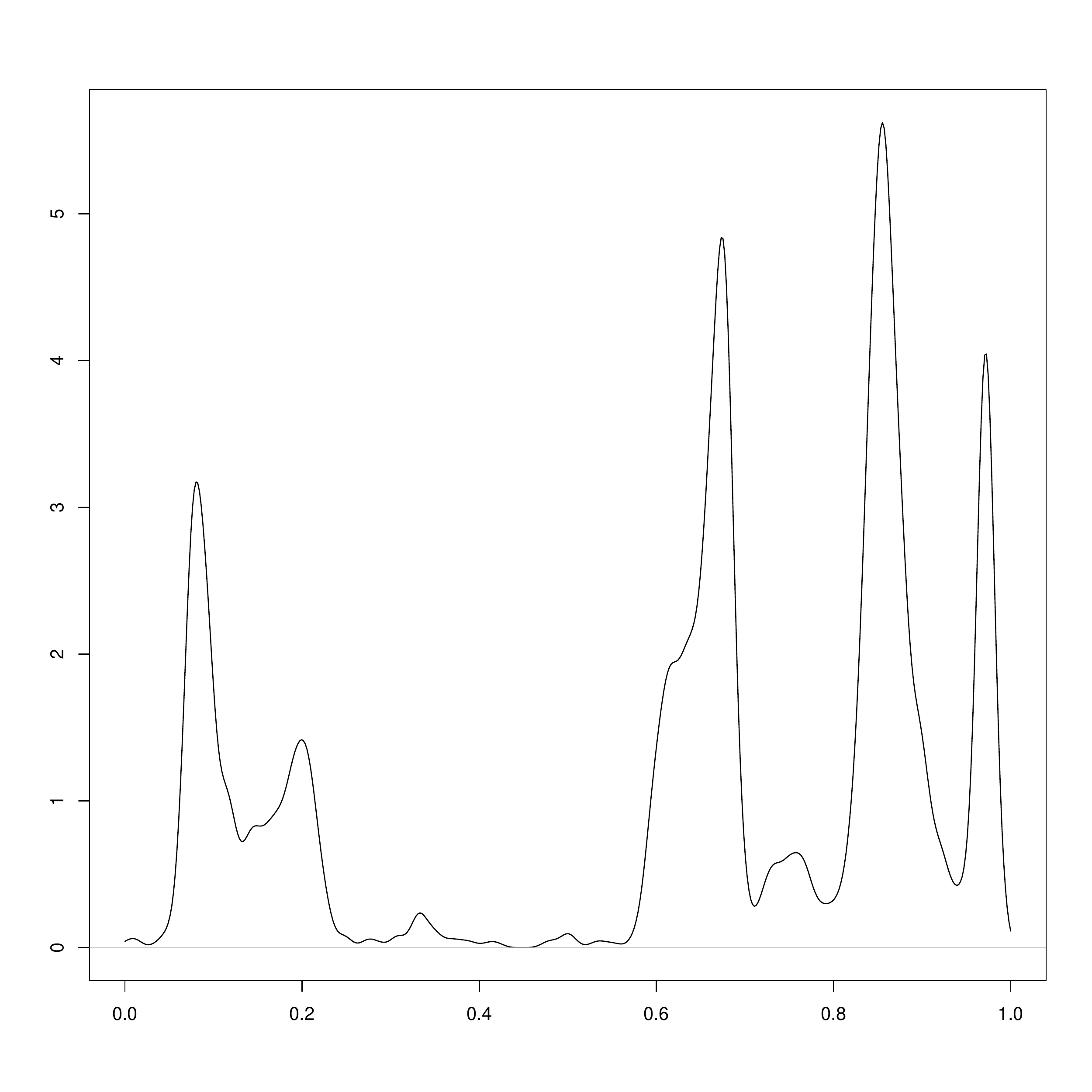}
\end{minipage}
\begin{minipage}{.45\textwidth}
  \centering
    \includegraphics[width=\textwidth]{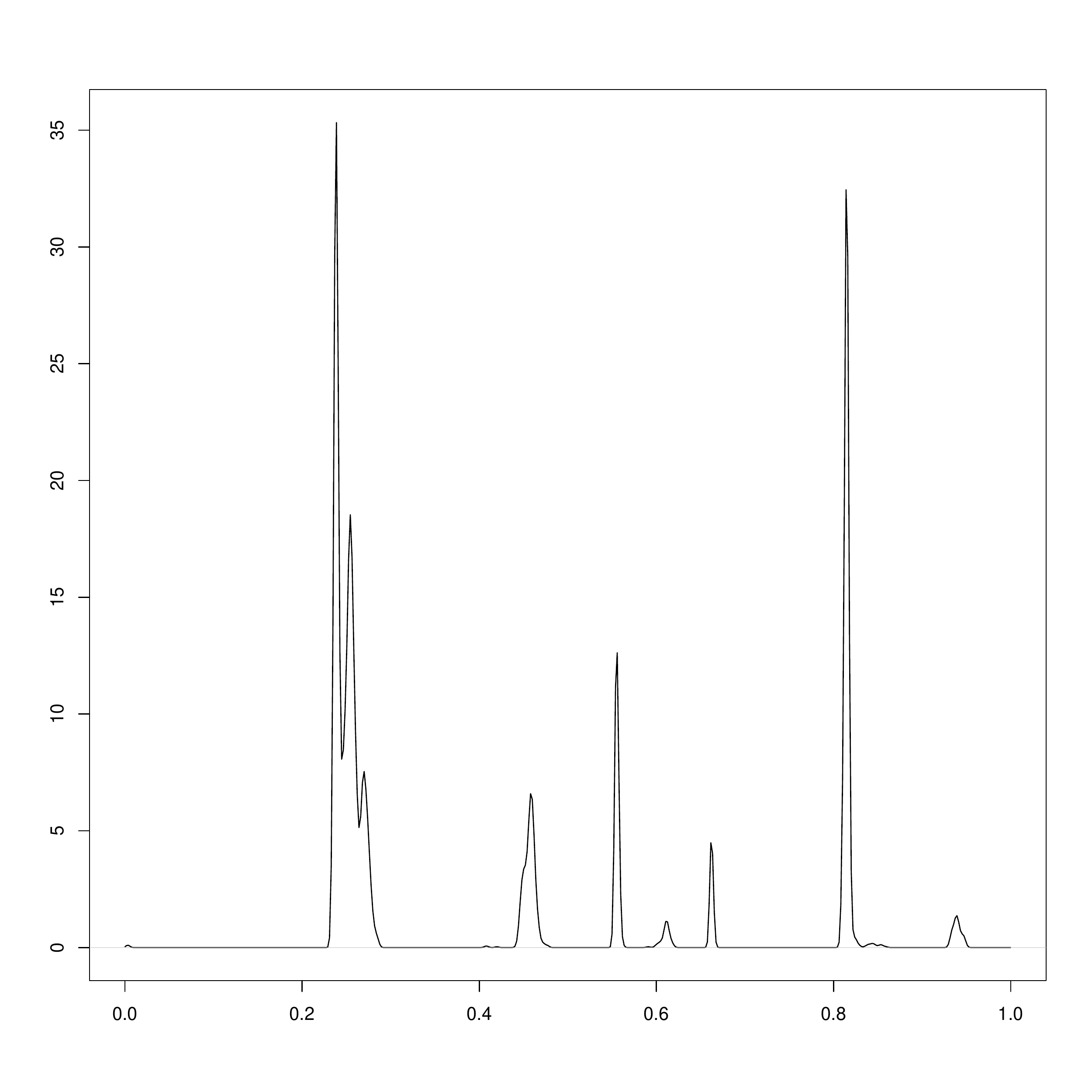}
\end{minipage}
 \caption{Simulated realizations of the density of  the kernel $K_{0,1}(0, \cdot)$ associated  with generators ${\mathcal G}^{N,1}$.   The parameters are $a=20, b=0.375$ in (a), and $a=60, b=0.125$ in (b). }
    \label{fig}
 \end{figure}

 As $n$ tends to infinity these flows of kernels converge to the  flow of kernels associated to a consistent family of sticky Brownian motions. Flows of this type were first considered by Le Jan and Raimond \cite{lejan4}. For a general splitting rule, they were defined by Howitt and Warren \cite{hw1}, and have subsequently been studied extensively in \cite{sss}.  In general the parameters of a consistent family of sticky Brownian motions can represented in terms of a splitting measure $\nu$  as
\begin{equation}
\theta(k:l)= \int_0^1 q^{k-1}(1-q)^{l-1} \nu(dq)
\end{equation}
For the parameters $\theta(k:l)$ given 
by Theorem 1, the measure $\nu$ is given by
\begin{equation}
\label{nu}
\nu(dq)=\frac{q(1-q)}{\phi(\Phi^{-1}(q))}dq
\end{equation}
where $\phi$ denotes the standard Gaussian density, and $\Phi$ the corresponding distribution function.
The  right and left speeds of the flow are defined by 
\begin{equation}
\beta_+= 2\int_0^1 q^{-1}\nu(dq) \text{ and } \beta_-=- 2\int_0^1 (1-q)^{-1}\nu(dq) 
\end{equation}
and with $\nu$ given by \eqref{nu} are both infinite.
Thus according to the Theorem 2.7 of \cite{sss}, the support of the  corresponding kernels is  almost surely equal to ${\mathbf R}$.
However, by Theorem 2.8 of \cite{sss}, for any $s \le t$ and $x$ the measure $K_{s,t}(x, \cdot)$ is purely atomic. This seems consistent with the simulations  which show  the mass becomeing more concentrated as the parameters $a$ and $b$  increase and decrease  respectively. It is less evident from these simulations that, in the limit, the set of points carrying the mass is  dense.
\vspace{.2in}

{\bf Acknowledgements.} This work was started during a visit to Universit\'{e}  Paris-Sud, and I would like to thanks the mathematics department there,  and Yves Le Jan in particular, for their hospitality. I'd also less to thank Peter Windridge for his help with writing the R code for the simulations.

\end{document}